\pdfoutput=1
\RequirePackage{ifpdf}
\ifpdf % We~are running pdfTeX in pdf mode
\documentclass[pdftex]{sigma}
\else
\documentclass{sigma}
\fi

\usepackage{titlesec}
\usepackage{thmtools}
\usepackage{thomasmacros}

\declaretheorem[name=Theorem, numberwithin=section]{thm}
\declaretheorem[name=Theorem, numbered=no]{thm*}

\declaretheorem[name=Lemma,numbered=no]{lem*}
\declaretheorem[name=Claim,numbered=no]{claim*}
\declaretheorem[name=Corollary,numberlike=thm]{cor}
\declaretheorem[name=Proposition,numberlike=thm]{prop}

\declaretheorem[name=Notation,numbered=no, style=remark]{nota*}

\declaretheorem[name=Construction,numbered=no, style=remark]{const*}

\declaretheorem[name=Question, numberlike=thm]{que}

\declaretheorem[name=Theorem]{prethm}

\numberwithin{equation}{section}

\usepackage[shortlabels]{enumitem}
\setlist[itemize]{noitemsep}
\setlist[enumerate]{label=\upshape(\arabic*)}
\newlist{myenumi}{enumerate}{1}
\setlist[myenumi,1]{label=\upshape(\roman*)}
\newlist{myenuma}{enumerate}{1}
\setlist[myenuma,1]{label=\upshape(\alph*)}

\usepackage[noabbrev]{cleveref} %
\crefname{figure}{Figure}{Figures}
\crefname{nota}{Notation}{Notations}
\crefname{table}{Table}{Tables}
\crefname{oque}{Open Question}{Open Questions}
\crefname{que}{Question}{Questions}
\crefname{thm}{Theorem}{Theorems}
\crefname{lem}{Lemma}{Lemmas}
\crefname{defi}{Definition}{Definitions}
\crefname{cor}{Corollary}{Corollaries}
\crefname{prop}{Proposition}{Propositions}
\crefname{ex}{Example}{Examples}
\crefname{rem}{Remark}{Remarks}
\crefname{const}{Construction}{Constructions}
\crefname{conj}{Conjecture}{Conjectures}
\crefname{section}{Section}{Sections}
\crefname{chapter}{Chapter}{Chapters}
\crefname{appendix}{Appendix}{Appendices}
\crefdefaultlabelformat{#2\textup{#1}#3}
\creflabelformat{enumi}{(#2#1#3)}
\crefname{prethm}{Theorem}{}
\crefname{precor}{Corollary}{}
\crefname{prelem}{Lemma}{}
\crefname{guide}{Guideline}{}
\newcommand{\ClSpinBdl}{\mathfrak{S}}
\NewDocumentCommand{\ClDirac}{}{\mathfrak{D}}
\newcommand{\ClDL}{\ClDirac_{\mathcal{L}}}

\newcommand{\DL}{\Dirac_{\mathcal{L}}}
\newcommand{\ue}{u_\epsilon}

\newcommand{\slfrac}[2]{\left.#1\middle/#2\right.}
\begin{document}

\allowdisplaybreaks

\newcommand{\arXivNumber}{2411.03882}

\renewcommand{\PaperNumber}{072}

\FirstPageHeading

\ShortArticleName{Ricci-Flat Manifolds, Parallel Spinors and the Rosenberg Index}

\ArticleName{Ricci-Flat Manifolds, Parallel Spinors\\ and the Rosenberg Index}

\Author{Thomas TONY}

\AuthorNameForHeading{T.~Tony}

\Address{Institute of Mathematics, University of Potsdam, Germany}
\Email{\href{mailto:math@ttony.eu}{math@ttony.eu}}
\URLaddress{\url{https://www.ttony.eu}}

\ArticleDates{Received May 19, 2025, in final form August 21, 2025; Published online August 25, 2025}

\Abstract{Every closed connected Riemannian spin manifold of non-zero $\hat{A}$-genus or non-zero Hitchin invariant with non-negative scalar curvature admits a parallel spinor, in particular is Ricci-flat. In this note, we generalize this result to closed connected spin manifolds of non-vanishing Rosenberg index. This provides a criterion for the existence of a parallel spinor on a finite covering and yields that every closed connected Ricci-flat spin manifold of dimension~$\geq 2$ with non-vanishing Rosenberg index has special holonomy.}

\Keywords{Ricci-flat manifolds; special holonomy; parallel spinor; scalar curvature; higher index theory}

\Classification{53C29; 58J20; 53C21; 58B34}

\section{Introduction and main results}

Let~$\brackets{M,g}$ be a connected Riemannian manifold of dimension~$n$ whose universal cover~$\tilde{M}$ admits a spin structure. We denote by~$\tilde{g}$ the pullback metric of~$g$ along the covering map~$\tilde{M}\to M$.
If $\bigl(\tilde{M},\tilde{g}\bigr)$ carries a parallel spinor, $g$ is Ricci-flat, i.e., its Ricci tensor vanishes identically. The converse of this statement holds for closed connected Riemannian spin manifolds with non-zero $\hat{A}$-genus or more general non-zero Hitchin invariant \cite[Section~4.2]{Hitchin1974}. This is a direct consequence of the Schr\"odinger--Lichnerowicz formula and the non-triviality of the kernel of the spin Dirac operator. It is an open question whether the converse of this statement holds in general.

\begin{que} \label{OpenQuestion}
 If $g$ is Ricci-flat, does~$\bigl(\tilde{M},\tilde{g}\bigr)$ admit a non-trivial parallel spinor?
\end{que}

In this note, we give a positive answer to Question~\ref{OpenQuestion} for closed connected spin manifolds with non-vanishing Rosenberg index (see Theorem~\ref{MainTheoremA}). As a consequence, in dimension~$\geq 2$ these manifolds have special holonomy for any Ricci-flat metric (see Corollary~\ref{CorollayB}).

The Rosenberg index of a closed spin manifold is a class in the real K-theory of the maximal group $\Cstar$-algebra of the fundamental group of the manifold. It is defined as the higher index of the $\ComplexCl_n$-linear spin Dirac operator twisted by the Mishchenko bundle (see Section~\ref{SectionRosenbergIndex}). In general, the Rosenberg index is not only non-vanishing whenever the $\hat{A}$-genus or the Hitchin invariant is non-zero, but also for many other classes of manifolds, as for example enlargeable manifolds~\cite{Hanke2006, Hanke2007}, and aspherical manifolds whose fundamental group satisfies the Novikov conjecture~\cite{Rosenberg1983}.
\begin{prethm} \label{MainTheoremA}
 Let~$\brackets{M,g}$ be a closed connected Riemannian spin manifold with non-vanishing Rosenberg index. Then the following statements are equivalent:
 \begin{enumerate}\itemsep=0pt
 \item \label{(1)}
 The Riemannian manifold~$\brackets{M,g}$ is Ricci-flat.
 \item \label{(2)}
 The universal cover~$\tilde{M}$, equipped with the pullback metric, admits a non-trivial parallel spinor.
 \item \label{(3)}
 There exists a finite Riemannian covering carrying a non-trivial parallel spinor with respect to the pullback spin structure.
 \item \label{(4)}
 The metric~$g$ has non-negative scalar curvature.
 \end{enumerate}
\end{prethm}

The main part of the proof of Theorem~\ref{MainTheoremA} is to show the implication \ref{(4)}$\Rightarrow$\ref{(3)}. The proof strategy is as follows. Suppose there exists a non-trivial twisted harmonic spinor. Then, by the Schr\"odinger--Lichnerowicz formula and the assumption on the scalar curvature, this twisted spinor is already parallel. We obtain a non-trivial parallel spinor on the universal covering of~$M$, which yields on a suitable finite covering a non-trivial parallel spinor with respect to the pullback spin structure. In the following paragraph, we outline the construction of such a~twisted harmonic spinor under the assumption~\ref{(4)} in Theorem~\ref{MainTheoremA}.

A~non-zero classical index such as the $\hat{A}$-genus or the Hitchin invariant gives rise to a non-trivial harmonic spinor. In general, a non-vanishing higher index such as the Rosenberg index does not give rise to a non-trivial harmonic spinor. But non-vanishing of the Rosenberg index is still an obstruction~--- so far the best obstruction based on the index theory of Dirac operators~--- to the existence of a positive scalar curvature metric on a closed connected spin manifold. A~classical rigidity statement by Bourguignon states that every closed Riemannian manifold that does not admit a positive scalar curvature metric is Ricci-flat if the scalar curvature is non-negative (see, e.g.,~\cite{Kazdan1975a}). Combining these two facts yield under assumption~\ref{(4)} that~$\brackets{M,g}$ is Ricci-flat. By the \textit{structure theorem for Ricci-flat manifolds} (see Theorem~\ref{ThmStructureTheorem}), the fundamental group of~$M$ contains the subgroup~$\Z^q$ with finite index, and it is possible to detect the non-vanishing Rosenberg index by the spin Dirac operator twisted by some flat finite-dimensional Hermitian vector bundle (see Proposition~\ref{PropDetectionPrinciple}). This detection principle we are using here is due to Schick and Wraith \cite{Schick2021a}, while an alternative approach is given in \cite[Section 3]{Ramras2013}. Here Ramras, Willett, and Yu \cite{Ramras2013} study the class of groups for which an associated non-vanishing higher index can be detected by a finite representation.

We now discuss the relation between Theorem~\ref{MainTheoremA} and special holonomy. A Riemannian manifold is called \textit{irreducible} if the representation of the reduced holonomy group on the general
linear group of the tangent space is irreducible. By Berger's holonomy list \cite{Berger1955}, the reduced holonomy group of any connected irreducible Ricci-flat manifold of dimension~$n$ is given by~$\SO\brackets{n}$ in the generic case, or one of the following subgroups:
$\SU\brackets{\slfrac{n}{2}}$ (locally Calabi--Yau), $\Sp\brackets{\slfrac{n}{4}}$ (locally hyper-K\"ahler), $ G_2 $ or $ \Spin\brackets{7}$. In dimension~$\geq 2$, the assumption that the Riemannian manifold is irreducible excludes all flat manifolds, in particular all Ricci-flat locally symmetric spaces \cite[Theorems 10.72 and 7.61]{Besse1987}. Note that the subgroups~$U\brackets{n}$ (locally K\"ahler) and~$\Sp\brackets{\slfrac{n}{4}}\cdot\Sp\brackets{1}$ (locally quaternionic K\"ahler) cannot occur as holonomy groups of Ricci-flat manifolds \cite[Proposition~10.29 and Theorem~14.45]{Besse1987}. While there exist examples of closed locally Calabi–Yau, locally hyper-K\"ahler, $G_2$, and $\mathrm{Spin}(7)$ manifolds, no closed irreducible Ricci-flat manifold with generic reduced holonomy is known. This leads to the following question.
\begin{que} \label{Question2}
 Does every closed connected Ricci-flat manifold of dimension~$n \geq 2$ have special holonomy, i.e., its reduced holonomy group is a proper subgroup of~$\SO\brackets{n}$?
\end{que}
The \textit{de Rham splitting theorem} states that every complete simply-connected Riemannian manifold is isometric to a~Riemannian product of complete non-flat irreducible manifolds and a~Euclidean space \cite[Theorem~10.43]{Besse1987}. We obtain that for reducible manifolds and flat manifolds the conclusion in Question~\ref{Question2} holds. Since every Ricci-flat manifold of dimension~$\leq 3$ is already flat, this includes all manifolds of dimension~$2$ and~$3$. In dimension~$\geq 4$, Question~\ref{Question2} is still open. It is related to spin geometry as follows. Every complete simply-connected irreducible Riemannian spin manifold~$\bigl(\tilde{M},\tilde{g}\bigr)$ admits a parallel spinor if and only if it has special holonomy (see \cite[Theorem~1.2]{Hitchin1974} and \cite{Wang1989}). We obtain by Theorem~\ref{MainTheoremA} the following corollary, which addresses Question~\ref{Question2} in the special case of spin manifolds with non-vanishing Rosenberg index.
\begin{cor} \label{CorollayB}
 Every closed connected Ricci-flat spin manifold of dimension~$\geq 2$ with non-vanishing Rosenberg index has special holonomy.
\end{cor}
\section{The Rosenberg index} \label{SectionRosenbergIndex}
In this section, we give a brief introduction to higher index theory with the goal to define the Rosenberg index \cite{Rosenberg1983}. See \cite{Lawson1989} for the construction of the ($\ComplexCl_n$-linear) spinor bundle, \cite{Lance1995} and \cite{Ebert2016} for some background concerning Hilbert $\Cstar$-modules and differential operators acting on Hilbert $\Cstar$-module bundles, \cite[Definition~2.1]{Tony2025} for a definition of graded real Dirac bundles carrying a~Hilbert $\Cstar$-module structure, and \cite{Roe} for an introduction to real K-theory of $\Cstar$-algebras and in particular for a~definition of the spectral picture of real K-theory.

Let~$\brackets{M,g}$ be a closed connected Riemannian spin manifold of dimension~$n$. This means that the manifold is oriented and equipped with a spin structure $\Spin\brackets{M,g}$, which is a lift of the $\SO\brackets{n}$-principal bundle of oriented orthonormal frames along the double covering $\Spin\brackets{n}\to \SO\brackets{n}$. The spin structure gives rise to the following two complex vector bundles
\begin{alignat*}{3}
& \SpinBdl\coloneqq \Spin\brackets{M,g}\times_{\rho} \Delta,\qquad&&
\rho\colon\ \Spin\brackets{n}\to \Aut\brackets{\Delta}\qquad
\text{complex spin representation},& \\
& \ClSpinBdl\coloneqq \Spin\brackets{M,g}\times_{\operatorname{cl}} \ComplexCl_n,\qquad&&
\operatorname{cl}\colon\ \Spin\brackets{n}\to \Aut\brackets{\ComplexCl_n}\qquad
\text{left multiplication}.
\end{alignat*}
They are called the \textit{irreducible spinor bundle} and the $\ComplexCl_n$-\textit{linear spinor bundle}, respectively. Here we denote by~$\ComplexCl_n$ the complexification of the Clifford algebra of~$\R^n$ which carries the structure of a graded real $\Cstar$-algebra and a graded real Hilbert $\ComplexCl_n$-module. The irreducible spinor bundle~$\SpinBdl$ equipped with its natural metric, connection and Clifford multiplication has the structure of a Dirac bundle. The $\ComplexCl_n$-linear spinor bundle~$\ClSpinBdl$, equipped with the grading, real structure and Hilbert module structure obtained fiber-wise from~$\ComplexCl_n$, the induced $\ComplexCl_n$-linear connection, and the Clifford multiplication, carries the structure of a graded real $\ComplexCl_n$-linear Dirac bundle. Note that~$\ClSpinBdl$ is isomorphic as a Dirac bundle to a direct sum of copies of the irreducible spinor bundle~$\SpinBdl$.

We denote by~$\pi$ the fundamental group of the manifold~$M$ and by $\Cstar \pi$ the maximal group $\Cstar$-algebra of~$\pi$. The \textit{Mishchenko bundle}~$\mathcal{L}$ is the associated bundle of the universal cover of~$M$ by the representation $l \colon \pi \to \Aut\brackets{\Cstar \pi}$ given by left multiplication \cite{Mishchenko1980}. It carries a natural flat connection, a $\Cstar \pi$-valued metric and a real structure. Tensoring the $\ComplexCl_n$-linear spinor bundle by the Mishchenko bundle yields a graded real $\ComplexCl_n\otimes \Cstar \pi$-linear Dirac bundle. The induced~${\ComplexCl_n \otimes \Cstar \pi}$-linear Dirac operator~$\ClDirac_{\mathcal{L}}$ extends to a self-adjoint and regular unbounded operator
\[
 \ClDL\colon\ \Hsp{1}{M,\ClSpinBdl\otimes \mathcal{L}} \To{} \Ltwo{M,\ClSpinBdl\otimes \mathcal{L}}.
\]
Here the Sobolev spaces~$\Hsp{1}{M,\ClSpinBdl\otimes \mathcal{L}}$ and $\Ltwo{M,\ClSpinBdl\otimes \mathcal{L}}$ are countably generated Hilbert $\ComplexCl_n\otimes \Cstar \pi$-modules which are similarly defined as the corresponding Sobolev spaces in the classical case. Since the manifold~$M$ is compact, the generalized Rellich theorem \cite[Theorem~2.33]{Ebert2016} yields that the functional calculus for unbounded regular and self-adjoint operators on Hilbert $\Cstar$-modules \cite[Theorem~1.19]{Ebert2016} restricts to
\begin{gather} \label{FunctionalCalculus}
 \Phi_{\ClDirac_{\mathcal{L}}}\colon\ \Cz{\R} \To{} \Kom_{\ComplexCl_n\otimes \Cstar \pi} \brackets{\Ltwo{M,\ClSpinBdl\otimes \mathcal{L}}}.
\end{gather}
This is stated in \cite[Proposition~2.34]{Ebert2016} with $g\equiv 1$ and~$X$ a point. Here~$\Kom_{\A}\brackets{E}$ denotes the compact operators on a Hilbert $\A$-module~$E$ for a~$\Cstar$-algebra~$\A$. The map~$\Phi_{\ClDirac_{\mathcal{L}}}$ defines a class in the spectral picture of the $n$-th real K-theory group of~$\Cstar \pi$ (see, e.g.,~\cite[Lecture 13]{Roe}). It is called the \textit{Rosenberg index} of the manifold~$M$.

\section{Proof of the main theorem} \label{SectionProofThmA}

In this section, we give a proof of Theorem~\ref{MainTheoremA}. It is based on the structure theorem for Ricci-flat manifolds \cite[Theorem~4.5]{Fischer1975}, which is a direct consequence of the Cheeger--Gromoll splitting theorem \cite{Cheeger1971}, and a detection principle for the spectrum of Hilbert $\Cstar \pi$-modules by Schick and Wraith \cite{Schick2021a}.

\begin{thm}[structure theorem for Ricci-flat manifolds] \label{ThmStructureTheorem}
 Every closed connected Ricci-flat manifold~$M$ admits a finite normal Riemannian covering $N\times T^q \to M$ consisting of a closed simply-connected Ricci-flat manifold~$N$ and a flat torus $T^q$. In particular, the fundamental group of~$M$ contains the subgroup~$\Z^q$ with finite index.
\end{thm}

\begin{prop}[{\cite[Proposition~3.6]{Schick2021a}}] \label{PropDetectionPrinciple}
Let~$\pi$ be a group containing the subgroup~$\Z^q$ with finite index~$d$, and let~$\mathcal{B}$ be a possibly unbounded self-adjoint and regular operator on a countably generated Hilbert $\Cstar \pi$-module. Suppose $T^2-1$ is compact in the sense of Hilbert $\Cstar \pi$-modules for the bounded transform \smash{$T\coloneqq \mathcal{B}\brackets{\mathcal{B}^2+1}^{-1/2}$}. If $0$ is in the spectrum of~$\mathcal{B}$, then there exists a~representation $\rho\colon \pi \to U\brackets{d}$ such that $0$ is in the spectrum of $\mathcal{B}\otimes_{\rho} \id$.
\end{prop}

\begin{proof}[Proof of Theorem~\ref{MainTheoremA}]
Let~$\brackets{M,g}$ be a closed connected Riemannian spin manifold of dimension~$n$ with non-vanishing Rosenberg index. We denote as in Section~\ref{SectionRosenbergIndex} the fundamental group of~$M$ by~$\pi$, the universal cover of~$M$ by~$\tilde{M}$, the Mishchenko bundle by~$\mathcal{L}$, the irreducible spinor bundle by~$\SpinBdl$, the $\ComplexCl_n$-linear spinor bundle by~$\ClSpinBdl$, and the induced Dirac operators by $\Dirac$ and $\ClDirac$, respectively. Moreover, we equip the universal cover of~$M$ by the pullback metric~$\tilde{g}$ and denote its irreducible spinor bundle by~$\tilde{\SpinBdl}$.

Since a parallel spinor on a finite covering of~$M$ lifts to a parallel spinor on the universal covering of~$M$ and the scalar curvature is by definition the trace of the Ricci curvature, the implications \ref{(3)}$\Rightarrow$\ref{(2)} and \ref{(1)}$\Rightarrow$\ref{(4)} are trivial. Suppose there exists a non-trivial parallel spinor~$u$ on~$\tilde{M}$. The curvature tensor~$K$ of the spinor bundle~$\tilde{\SpinBdl}$ is related to the Ricci curvature via
\begin{equation} \label{RicCurvatureRealtion}
 \Ric_{\tilde{g}} \brackets{X} \cdot u = -2 \sum_{j=1}^n e_j\cdot K_{X,e_j}u
\end{equation}
for any local orthonormal frame $e_1,\dots,e_n$ of~$T\tilde{M}|_U$ and any vector field~$X$ supported in an open subset~$U$ \cite[Section 3.1]{Friedrich2000}. Since~$u$ is non-trivial and parallel, this yields that~$\tilde{M}$ is Ricci-flat, hence~$M$ is Ricci-flat. This proves the implication \ref{(2)}$\Rightarrow$\ref{(1)}. It remains to show the implication \ref{(4)}$\Rightarrow$\ref{(3)}. Suppose the scalar curvature of the Riemannian manifold~$\brackets{M,g}$ is non-negative. We~have to show that there exists a finite Riemannian covering of~$M$ together with a parallel spinor with respect to the pullback spin structure. The proof splits into the following four steps.

{\it Step $1$: The Riemannian manifold~$\brackets{M,g}$ is Ricci-flat.} %\label{Step1}
Since the Rosenberg index does not vanish, the manifold~$M$ does not admit a positive scalar curvature metric. Since the scalar curvature of~$M$ is non-negative, it follows by Ricci flow or a classical result by Bourguignon~\cite{Kazdan1975a} that the manifold~$M$ is Ricci-flat. A new proof using only spinorial techniques is given by the author in \cite[Sections~3 and~4]{Tony2025} and will be briefly sketched here.

The non-vanishing Rosenberg index gives rise to a family of almost $\ClDL$-harmonic sections $\{\ue\}_{\epsilon>0}$, i.e., $\ue\in \Cinfty{M,\ClSpinBdl\otimes \mathcal{L}}$, $\Ltwonorm{\ue}=1$ and $\Ltwonorm[\large]{\ClDL^j \ue}<\epsilon$ for all $j>0$ and all $\epsilon>0$ \cite[Lemma~3.2]{Tony2025}. Since the scalar curvature of~$M$ is non-negative, we obtain by the Schr\"odinger--Lichnerowicz formula $\Ltwonorm{\nabla \ue}<\epsilon$ for all $\epsilon>0$. We can improve this estimate by Moser iteration~\cite[Lemma~3.4~ff.]{Tony2025} and obtain $\Inftynorm{\nabla \ue}\lesssim \epsilon^r$ for all $\epsilon\in \brackets{0,1}$ and a suitable positive constant $r$ which is independent of~$\epsilon$. By the Poincar\'e inequality, we obtain that the family is almost constant, i.e.,
\[
 \norm[\big]{\bar{u}_{\epsilon}-\scalarproduct[\big]{\ue\brackets{p}}{\ue\brackets{p}}}_{\ComplexCl_n\otimes \Cstar \pi}\lesssim \epsilon^r \qquad \text{with} \quad
 \bar{u}_\epsilon \coloneqq \frac{1}{\vol\brackets{M}}\int_M \scalarproduct[\big]{\ue\brackets{q}}{\ue\brackets{q}} \intmathd q,
\]
for all $\epsilon \in \brackets{0,1}$ and all $p\in M$. Furthermore, we can choose the family $\{\ue\}_{\epsilon>0}$ such that $\bigl\|\nabla^2\ue\bigr\|\lesssim \sqrt{\epsilon}$ holds for all $\epsilon>0$ \cite[Lemma~4.3]{Tony2025}. Together with equation~\eqref{RicCurvatureRealtion}, we obtain
\begin{equation} \label{EquationRicciFlat}
 \Ltwonorm{\Ric \brackets{X}} \lesssim \Ltwonorm{\Ric \brackets{X}\cdot \ue} \lesssim \sqrt{\epsilon}
\end{equation}
for any vector field~$X$ on~$M$ and all sufficiently small $\epsilon>0$, hence~$M$ is Ricci-flat.

{\it Step $2$: There exists a flat Hermitian bundle over~$M$ carrying a non-zero twisted parallel~spinor.} %\label{Step2}
Since the manifold~$M$ is Ricci-flat by Step~1, we can apply Theorem~\ref{ThmStructureTheorem} and the fundamental group of~$M$ contains the subgroup $\Z^q$ with finite index~$d$ for some~$q$ and~$d$. The Rosenberg index of~$M$ does not vanish by assumption, hence the twisted Dirac operator
\[
\ClDL\colon\ \Hsp{1}{M,\ClSpinBdl\otimes \mathcal{L}} \To{} \Ltwo{M,\ClSpinBdl\otimes \mathcal{L}}
\]
is not invertible (see, e.g., \cite[Lemma 2.4\,(ii)]{Tony2025}). Since the $\ComplexCl_n$-linear spinor bundle~$\ClSpinBdl$ is a direct sum of copies of the irreducible spinor bundle~$\SpinBdl$, we obtain that the twisted Dirac operator
\[
\DL\colon\ \Hsp{1}{M,\SpinBdl\otimes \mathcal{L}} \To{} \Ltwo{M,\SpinBdl\otimes \mathcal{L}}
\]
is also not invertible. The functional calculus associated to~$\DL$ restricts as in equation~\ref{FunctionalCalculus}, hence the operator $T^2-1$ is compact for the bounded transform \smash{$T\coloneqq \DL\brackets{\DL^2+1}^{-1/2}$}, and we can apply Proposition~\ref{PropDetectionPrinciple}. This yields a representation $\rho\colon \pi\to U\brackets{d}$ such that $\DL\otimes_{\rho} \id$ is not invertible. Note that the representation~$\rho$ extends by the universal property of the maximal group $\Cstar$-algebra to a representation $\overline{\rho}\colon \Cstar \pi \to \GL_d\brackets{\C}$. We define~$V$ as the flat Hermitian bundle $\tilde{M}\times_{\rho} \C ^d$. Following the identifications
\[
\brackets{\SpinBdl\otimes \mathcal{L}} \otimes_{\overline{\rho}} \C^d\cong\SpinBdl\otimes
\bigl(\tilde{M} \times_{\pi} \Cstar \pi \otimes_{\overline{\rho}} \C^d\bigr) \cong \SpinBdl \otimes V,
\]
we obtain $\DL\otimes_{\overline{\rho}} \id=\Dirac_V$ (compare \cite[Proposition~3.1]{Schick2021a}), hence the twisted Dirac operator~$\Dirac_V$ is not invertible. Since the manifold~$M$ is closed and the twisting bundle~$V$ has finite-dimensional fibers, the spectrum of~$\Dirac_V$ consists of eigenvalues only, hence the kernel of $\Dirac_V$ is non-zero. Let~$u$ be a non-trivial element in the kernel of~$\Dirac_V$. Since the scalar curvature of~$M$ is non-negative by assumption, we obtain by the Schr\"odinger--Lichnerowicz formula
\[
 \Ltwonorm{\nabla u}^2\leq \Ltwoscalarproduct{\nabla^*\nabla u}{u}+\tfrac{1}{4}\Ltwoscalarproduct{\scal_g u}{u}= \Ltwonorm{\Dirac_V u}^2=0.
\]
This shows that~$u$ is parallel, and the second step is proved.

{\it Step $3$: There exists a non-trivial parallel spinor on~$\tilde{M}$ equipped with the pullback metric.} %\label{Step3}
This step corresponds to \cite[Proposition~3.3]{Schick2021a}. The proof is given below for the sake of completeness. Let~$V$ be the flat Hermitian bundle over~$M$ and~$u$ be the twisted parallel spinor constructed in Step~2. We denote by~$\tilde{V}$ and~$\tilde{u}$ the pullback of~$V$ and~$u$ to the universal covering of~$M$, respectively. By construction the twisted spinor $\tilde{u}$ is parallel with respect to the pullback connection. Since the covering map $\tilde{M}\to M$ is a local isometry, hence preserves all local structures, the pullback connection on~$\tilde{M}$ equals the usual spinor connection twisted with the pullback connection on~$\tilde{V}$. Since $V$ is by definition in Step~2 equal to $\tilde{M}\times_{\rho} \C^d$ with the representation $\rho\colon \pi \to U\brackets{d}$, we obtain
\[
 \tilde{V}=\tilde{M}\times_{\tilde{\rho}}\C^d\cong\tilde{M}\times \C^d \qquad \text{with} \quad \tilde{\rho}\colon\ \{1\}=\pi_1\brackets[\big]{\tilde{M}}\To{} \pi \To{\rho} U\brackets{d}.
\]
This yields an identification \smash{$\tilde{\SpinBdl}\otimes \tilde{V} \cong \bigl(\tilde{\SpinBdl}\bigr)^d$} as bundles with connection, hence the parallel twisted spinor $\tilde{u}$ can be identified with a vector of~$d$ parallel spinors on~$\tilde{M}$. Since~$\tilde{u}$ is non-trivial, at least one component of~$\tilde{u}$ is non-trivial and we obtain a non-trivial parallel spinor on~$\tilde{M}$. This proves the third step.

{\it Step $4$: There exists a finite Riemannian covering of~$M$ together with a non-trivial parallel spinor with respect to the pullback spin structure.} %\label{Step4}
In the first part of the argument bellow, we follow the approach taken in the proof of \cite[Lemma 2.7]{Schick2021a}. By the structure theorem for Ricci-flat manifolds (see Theorem~\ref{ThmStructureTheorem}), there exists a finite Riemannian covering~$\brackets{N,g_N}\times \brackets{T^q,g_{\text{fl}}} \to \brackets{M,g}$ where~$\brackets{N,g_N}$ is a simply-connected Ricci-flat manifold and~$T^q$ the $q$-torus equipped with a flat metric~$g_{\text{fl}}$. This yields that the universal covering of~$M$ is isometric to $\brackets{N,g_N}\times \brackets{\R^q,\tilde{g}_{\text{fl}}}$. Note that a Riemannian product admits a non-trivial parallel spinor if and only if each of its factors does. Since the universal covering of~$M$ carries a non-trivial parallel spinor by Step~3, $\brackets{N,g_N}$ carries a non-trivial parallel spinor. It now suffices to show that a suitable double covering of the $q$-torus equipped with a flat metric and the spin structure pulled back from~$M$ carries a non-trivial parallel spinor.

Since the $q$-torus is parallelizable and equipped with a flat metric, the $\SO\brackets{q}$-principal bundle of orthonormal frames~$\SO\brackets{T^q}$ can be trivialized such that the trivial connection coincides with the Levi-Civita connection. Let~$s$ be the corresponding global section of~$\SO\brackets{T^q}$. Note that double coverings over a connected manifold~$X$ are classified by homotopy classes of homomorphisms~${\pi_1\brackets{X}\to \Z_2}$ \cite[Section 1.3]{Hatcher2001}.

Let~$G$ be $\set{0}$, $\Z$ or~$\Z_2$ for~$q=1$, $q=2$ or~$q>2$, respectively. By the trivialization of $\SO\brackets{T^q}$, we obtain an isomorphism $\pi_1\brackets{\SO\brackets{T^q}}\cong \Z^q\times G$ and the spin structure on~$T^q$, which is a certain double cover of~$\SO\brackets{T^q}$, is classified by some map
\[
 \brackets{f_1,f_2}\colon\ \Z^q\times G \to \Z_2.
\]
For~$n>1$, the condition to be a spin principal bundle yields~$f_2\neq 0$ \cite[Section~1, Corollary 1.5]{Lawson1989}.
The isomorphism $\pi_1\brackets{T^q}\cong \Z^q$ together with the map~$f_1\colon \Z^q \to \Z_2$ gives rise to a double covering~$\bar{T}^q \to T^q$. By construction, the pullback of the spin structure of~$T^q$ along the map~${\bar{T}^q\to T^q}$, denoted by ~$\Spin\bigl(\bar{T}^q\bigr)$, is classified by the map $\brackets{0,f_2}\colon \Z^q\times G \to \Z_2$. Therefore, the pullback of the section~$s$ along the double covering~$\bar{T}^q\to T^q$ lifts to a section of~$\Spin\bigl(\bar{T}^q\bigr)$, hence yields a~trivialization of the $\Spin\brackets{q}$-principal bundle of~$\bar{T}^q$.

Since the trivial connection on the trivialization of~$\SO\bigl(\bar{T}^q\bigr)$ coincides with the Levi-Civita connection and we just lifted this trivialization, the trivial connection on the trivialization of~$\Spin\bigl(\bar{T}^q\bigr)$ also coincides with the Levi-Civita connection. It follows that the constant sections of the associated trivialized spinor bundle over~$\bar{T}^q$ are parallel, and Theorem~\ref{MainTheoremA} is proved.
\end{proof}

\subsection*{Acknowledgements}

I thank Bernd Ammann for coming up with this interesting question and his hospitality during my short-term visit in Regensburg at the SFB 1085 Higher Invariants. I also thank my advisor Rudolf Zeidler for his continuous support, as well as Thorsten Hertl and the anonymous referees for their valuable comments.

Funded by the European Union (ERC Starting Grant 101116001~-- COMSCAL)\footnote{Views and opinions expressed are however those of the author(s) only and do not necessarily reflect those of the European Union or the European Research Council. Neither the European Union nor the granting authority can be held responsible for them.}
 and by the Deutsche Forschungsgemeinschaft (DFG, German Research Foundation)~-- Project-ID 427320536 -- SFB 1442, as well as under Germany’s Excellence Strategy EXC 2044 390685587, Mathematics M\"unster: Dynamics--Geometry--Structure.

\pdfbookmark[1]{References}{ref}
\LastPageEnding


\begin{thebibliography}{99}
\footnotesize\itemsep=0pt

\bibitem{Berger1955}
Berger M., Sur les groupes d'holonomie homog\`ene des vari\'et\'es \`a
 connexion affine et des vari\'et\'es riemanniennes,
 \href{https://doi.org/10.24033/bsmf.1464}{\textit{Bull. Soc. Math. France}}
 \textbf{83} (1955), 279--330.

\bibitem{Besse1987}
Besse A.L., Einstein manifolds, \textit{Ergeb. Math. Grenzgeb.~(3)}, Vol.~10,
 \href{https://doi.org/10.1007/978-3-540-74311-8}{Springer}, Berlin, 1987.

\bibitem{Cheeger1971}
Cheeger J., Gromoll D., The splitting theorem for manifolds of nonnegative
 {R}icci curvature,
 \href{https://doi.org/10.4310/jdg/1214430220}{\textit{J.~Differential
 Geometry}} \textbf{6} (1971), 119--128.

\bibitem{Ebert2016}
Ebert J., Elliptic regularity for Dirac operators on families of noncompact
 manifolds, \href{http://arxiv.org/abs/1608.01699}{arXiv:1608.01699}.

\bibitem{Fischer1975}
Fischer A.E., Wolf J.A., The structure of compact {R}icci-flat {R}iemannian
 manifolds,
 \href{https://doi.org/10.4310/jdg/1214432794}{\textit{J.~Differential
 Geometry}} \textbf{10} (1975), 277--288.

\bibitem{Friedrich2000}
Friedrich T., Dirac operators in {R}iemannian geometry, \textit{Grad. Stud.
 Math.}, Vol.~25, \href{https://doi.org/10.1090/gsm/025}{American Mathematical
 Society}, Providence, RI, 2000.

\bibitem{Hanke2006}
Hanke B., Schick T., Enlargeability and index theory,
 \href{https://doi.org/10.4310/jdg/1175266206}{\textit{J.~Differential Geom.}}
 \textbf{74} (2006), 293--320,
 \href{http://arxiv.org/abs/math.GT/0403257}{arXiv:math.GT/0403257}.

\bibitem{Hanke2007}
Hanke B., Schick T., Enlargeability and index theory: infinite covers,
 \href{https://doi.org/10.1007/s10977-007-9004-3}{\textit{$K$-Theory}}
 \textbf{38} (2007), 23--33,
 \href{http://arxiv.org/abs/math.GT/0604540}{arXiv:math.GT/0604540}.

\bibitem{Hatcher2001}
Hatcher A., Algebraic topology, Cambridge University Press, Cambridge, 2002.

\bibitem{Hitchin1974}
Hitchin N., Harmonic spinors,
 \href{https://doi.org/10.1016/0001-8708(74)90021-8}{\textit{Adv. Math.}}
 \textbf{14} (1974), 1--55.

\bibitem{Kazdan1975a}
Kazdan J.L., Warner F.W., Scalar curvature and conformal deformation of
 {R}iemannian structure,
 \href{https://doi.org/10.4310/jdg/1214432678}{\textit{J.~Differential
 Geometry}} \textbf{10} (1975), 113--134.

\bibitem{Lance1995}
Lance E.C., Hilbert {$C^*$}-modules. A toolkit for operator algebraists,
 \textit{London Math. Soc. Lecture Note Ser.}, Vol.~210,
 \href{https://doi.org/10.1017/CBO9780511526206}{Cambridge University Press},
 Cambridge, 1995.

\bibitem{Lawson1989}
Lawson Jr. H.B., Michelsohn M.L., Spin geometry, \textit{Princeton Math. Ser.},
 Vol.~38, Princeton University Press, Princeton, NJ, 1989.

\bibitem{Mishchenko1980}
Mishchenko A.S., Fomenko A.T., The index of elliptic operators over {$C^{\ast}
 $}-algebras, \textit{Math. USSR. Izv.} \textbf{15} (1980), 87--112.

\bibitem{Ramras2013}
Ramras D., Willett R., Yu G., A finite-dimensional approach to the strong
 {N}ovikov conjecture,
 \href{https://doi.org/10.2140/agt.2013.13.2283}{\textit{Algebr. Geom.
 Topol.}} \textbf{13} (2013), 2283--2316,
 \href{http://arxiv.org/abs/1203.6168}{arXiv:1203.6168}.

\bibitem{Roe}
Roe J., Lectures on K-theory and operator algebras, {AMS} Open Math. Notes,
 2017,
 \url{https://www.ams.org/open-math-notes/omn-view-listing?listingId=110719}.

\bibitem{Rosenberg1983}
Rosenberg J., {$C^{\ast} $}-algebras, positive scalar curvature, and the
 {N}ovikov conjecture, \href{https://doi.org/10.1007/BF02953775}{\textit{Publ.
 Math. Inst. Hautes Etudes Sci.}} \textbf{58} (1983), 409--424.

\bibitem{Schick2021a}
Schick T., Wraith D.J., Non-negative versus positive scalar curvature,
 \href{https://doi.org/10.1016/j.matpur.2020.09.010}{\textit{J.~Math. Pures
 Appl.~(9)}} \textbf{146} (2021), 218--232,
 \href{http://arxiv.org/abs/1607.00657}{arXiv:1607.00657}.

\bibitem{Tony2025}
Tony T., Scalar curvature rigidity and the higher mapping degree,
 \href{https://doi.org/10.1016/j.jfa.2024.110744}{\textit{J.~Funct. Anal.}}
 \textbf{288} (2025), 110744, 41~pages,
 \href{http://arxiv.org/abs/2402.05834}{arXiv:2402.05834}.

\bibitem{Wang1989}
Wang M.Y., Parallel spinors and parallel forms,
 \href{https://doi.org/10.1007/BF00137402}{\textit{Ann. Global Anal. Geom.}}
 \textbf{7} (1989), 59--68.

\end{thebibliography}
\end{document}